\documentclass[a4paper,12pt]{article}
\usepackage{fancyhdr,graphicx}
\usepackage{amsfonts}
\usepackage{amssymb}
\usepackage{amsthm}
\usepackage{newlfont}
\usepackage{amsmath}
\usepackage[top=2.5cm,bottom=2.5cm,left=2.5cm,right=2.5cm]{geometry}
\newtheorem{thm}{Theorem}[section]
\newtheorem{Lemma}[thm]{Lemma}

\usepackage{latexsym,bm}
\title{\vspace{0cm}{Prescribed matchings extend to Hamiltonian cycles in hypercubes with faulty edges\thanks{This
work is supported by NSFC (grant no. 61073046).}}}

\author{Fan Wang, Heping Zhang\footnote{Corresponding author.} \\
\small{School of Mathematics and Statistics, Lanzhou University,
Lanzhou, Gansu 730000, P. R. China}
\\\small{E-mail addresses: wangfan2009@lzu.edu.cn, zhanghp@lzu.edu.cn}}

\date{}

\begin{document}

\maketitle

\thispagestyle{empty}

\begin{abstract}
Ruskey and Savage asked the following question: Does every matching
of $Q_{n}$ for $n\geq2$ extend to a
Hamiltonian cycle of $Q_{n}$? J. Fink showed that the
question is true for every perfect matching, and solved the Kreweras' conjecture. In this paper
we consider the question in
hypercubes with faulty edges. We show that every matching $M$ of at
most $2n-1$ edges can be extended to a Hamiltonian cycle of $Q_{n}$
for $n\geq2$. Moreover, we can prove that when $n\geq4$ and $M$ is
nonempty this result
still holds even if $Q_{n}$ has at most
$n-1-\lceil\frac{|M|}{2}\rceil$ faulty edges with one exception.

%\vspace{0.3cm}
\end{abstract}

%%introduction--------------------------------------------------
\section{Introduction}
The $n$-dimensional hypercube $Q_{n}$ is one of the most popular and
efficient interconnection networks. There is a large amount of
literature on graph-theoretic properties of hypercubes as well as on
their applications in parallel computing (e.g., see \cite{8,9}).

It is well known that $Q_{n}$ is Hamiltonian for every $n\geq2$.
This statement dates back to 1872 \cite{10}. Since then, the
research on Hamiltonian cycles in hypercubes satisfying certain
additional properties has received considerable attention. The
applications in parallel computing inspired the study of Hamiltonian
cycles in hypercubes with faulty edges \cite{11,12,6}.
Dvo\v{r}\'{a}k \cite{5} showed that any set of at most $2n-3$ edges
of $Q_{n}(n\geq2)$ that induces vertex-disjoint paths is contained
in a Hamiltonian cycle. Wang et al. \cite{6} proved that this result
still holds even if $Q_{n}$ has some faulty edges, see Lemma
\ref{cycle} below. More details about this topic see \cite{14,3}.

Kreweras \cite{2} conjectured that every perfect matching of $Q_{n}$
for $n\geq2$ can be extended to a Hamiltonian cycle of $Q_{n}$. In
\cite{4, F} Fink solved this conjecture by proving a stronger
result. Ruskey and Savage \cite{15} asked the following question:
Does every matching of $Q_{n}$ for $n\geq2$ extend to a Hamiltonian
cycle of $Q_{n}$? Fink \cite{4} pointed out that the statement is
true for $n=2,3,4$. The result in \cite{5} implied that every
matching of at most $2n-3$ edges can be extended to a Hamiltonian
cycle of $Q_{n}$. Vandenbussche and West \cite{20} showed that every
k-suitable matching of at most $k(n-k)+\frac{(k-1)(k-2)}{2}$ edges
for $1\leq k\leq n-3$ and every induced matching can be extended to
a perfect matching of $Q_{n}$, so can be extended to a Hamiltonian
cycle of $Q_{n}$.

In this paper, we consider Ruskey and Savage's question in faulty
hypercubes and obtain the following main results: every matching $M$
of at most $2n-1$ edges can be extended to a Hamiltonian cycle of
$Q_{n}$ for $n\geq2$;  when $n\geq4$ and $M$ is nonempty this result
still holds even if $Q_{n}$ has at most
$n-1-\lceil\frac{|M|}{2}\rceil$ faulty edges with one exception. The
rest of this paper is organized as follows. In Section 2 we
introduce some necessary definitions and preliminaries. In Section 3
and Section 4 we discuss bases of induction of the two main
theorems. The main results are stated and proved in Section 5.

\section{Definitions and preliminaries}
The terminology and notation used in this paper but undefined below
can be found in \cite{1}. As usual, the vertex set and edge set of a
graph $G$ are denoted by $V(G)$ and $E(G)$. For a set $E\subseteq
E(G)$, let $G-E$ denote the graph with vertices $V(G)$ and edges
$E(G)\setminus E$. The distance between two vertices $u$ and $v$ is
the number of edges in a shortest path between $u$ and $v$ in $G$,
denoted by $d_{G}(u,v)$, with the subscript being omitted when the
context is clear. For any two edges $uv,xy$,
$d(uv,xy)=\min\{d(u,x),d(u,y),d(v,x),d(v,y)\}$. Throughout the
paper, $n$ always denotes a positive integer while $[n]$ denotes the
set $\{1,2,...,n\}$.

The $n$-$dimensional$ $hypercube$ $Q_{n}$ is a graph whose vertex
set consists of all binary strings of length $n$, with two vertices
being adjacent whenever the corresponding strings differ in exactly
one position. An edge in $Q_{n}$ is called an $i$-$dimensional$
$edge$ if its endvertices differ in the $i$th position. The set of
all $i$-dimensional edges of $Q_{n}$ is denoted by $E_{i}$. For any
given $j\in[n]$, let $Q_{n-1}^{0}$ and $Q_{n-1}^{1}$ be two
$(n-1)$-dimensional subcubes of $Q_{n}$ induced by all the vertices
with the $j$th positions being 0 and 1, respectively. Since
$Q_{n}-E_{j}=Q_{n-1}^{0}\cup Q_{n-1}^{1}$, we say that $Q_{n}$ is
decomposed into two $(n-1)$-dimensional subcubes $Q_{n-1}^{0}$ and
$Q_{n-1}^{1}$ by $E_{j}$. Any vertex $u\in V(Q_{n-1}^{0})$ has in
$Q_{n-1}^{1}$ a unique neighbor, denoted by $u_{1}$. Similarly, any
vertex $v\in V(Q_{n-1}^{1})$ has in $Q_{n-1}^{0}$ a unique neighbor,
denoted by $v_{0}$. For any edge $e=uv\in E(Q_{n-1}^{0})$, $e_{1}$
denotes the edge $u_{1}v_{1}\in E(Q_{n-1}^{1})$. For given
$F\subseteq E(Q_{n})$, let $F_{\delta}=F\cap E(Q_{n-1}^{\delta})$
for $\delta\in\{0,1\}$.

\begin{Lemma}\cite{17}\label{decompose} For $n\geq2$, let $e$ and $f$ be two
disjoint edges in $Q_{n}$. Then $Q_{n}$ can be decomposed into two
$(n-1)$-dimensional subcubes such that one contains $e$ and the
other contains $f$.
\end{Lemma}

Let us recall the following classical result, originally proved by
Havel in \cite{Havel1}.

\begin{Lemma}\cite{Havel1}\label{Havel} Let $n\geq1$ and $x,y\in V(Q_{n})$
be such that $d(x,y)$ is odd. Then there exists a Hamiltonian path
between $x$ and $y$ in $Q_{n}$.
\end{Lemma}

\begin{Lemma}\cite{5}\label{pac} For $n\geq2$, let $x,y\in V(Q_{n})$
and $e\in E(Q_{n})$ such that $d(x,y)$ is odd and $e\neq xy$. Then
there is a Hamiltonian path of $Q_{n}$ between $x$ and $y$ passing
through edge $e$.
\end{Lemma}

\begin{Lemma}\cite{7}\label{pad} For $n\geq3$, let $u,v\in V(Q_{n})$ and
$F\subseteq E(Q_{n})$ such that $d(u,v)$ is odd and $|F|\leq1$. Then
there exists a Hamiltonian path in $Q_{n}-F$ between $u$ and $v$.
\end{Lemma}

A path with endvertices $u$ and $v$ is denoted by $P_{uv}$. We say
that paths $\{P_{i}\}_{i=1}^{k}$ are $spanning$ $paths$ of a graph
$G$ if $\{V(P_{i})\}_{i=1}^{k}$ partitions $V(G)$.

\begin{Lemma}\cite{5}\label{twopa} For $n\geq2$, let $x,y,u,v$ be pairwise
distinct vertices of $Q_{n}$ such that both $d(x,y)$ and $d(u,v)$
are odd. Then (i) there exist spanning paths $P_{xy},P_{uv}$ of
$Q_{n}$; (ii) moreover, in the case when $d(x,y)=1$, path $P_{xy}$
can be chosen such that $P_{xy}=xy$, unless $n=3$, $d(u,v)=1$ and
$d(xy,uv)=2$.
\end{Lemma}

Note that for any edge $e\in E(Q_{3})$ there exists a unique edge
$f\in E(Q_{3})$ such that $d(e,f)=2$.

A forest is $linear$ if each component of it is a path.

\begin{Lemma}\cite{5}\label{forest} For $n\geq2$, let $E\subseteq E(Q_{n})$
with $|E|\leq2n-3$. Then there exists a Hamiltonian cycle of $Q_{n}$
passing through $E$ if and only if the subgraph induced by $E$ is a
linear forest.
\end{Lemma}

\begin{Lemma}\cite{13}\label{cyc} Let $n\geq3$ and $F\subseteq E(Q_{n})$
with $|F|\leq n-2$. Then any edge of $E(Q_{n})\setminus F$ lies on a
Hamiltonian cycle of $Q_{n}-F$.
\end{Lemma}

\begin{Lemma}\cite{6}\label{cycle} For $n\geq2$, let $F\subseteq E(Q_{n})$,
$E\subseteq E(Q_{n})\setminus F$ with $1\leq |E|\leq2n-3$, $|F|\leq
n-2-\lfloor\frac{|E|}{2}\rfloor$. If the subgraph induced by $E$ is
a linear forest, then all edges of $E$ lie on a Hamiltonian cycle of
$Q_{n}-F$.
\end{Lemma}

\section{Base of induction of the first main theorem}

A set of edges in a graph $G$ is called a $matching$ if no two edges
have a point in common. A matching is $perfect$ if it covers all of
$V(G)$.

\begin{Lemma}\cite{4,F}\label{Fink} For every perfect matching $M$ of $K(Q_{n})$
there exists a perfect matching $R$ of $Q_{n}$, $n\geq2$, such that
$M\cup R$ is a Hamiltonian cycle of $K(Q_{n})$, where $K(Q_{n})$ is
the complete graph on the vertices of the hypercube $Q_{n}$.
\end{Lemma}

\begin{Lemma}\label{3path} Let $M$ be a matching of $Q_{3}$
and $u,v\in V(Q_{3})$ such that $u\notin V(M)$ and
$d(u,v)$ is odd. Then there exists a Hamiltonian path of $Q_{3}$
between $u$ and $v$ passing through $M$.
\end{Lemma}

\begin{proof} Since $u\notin V(M)$, we have $|M|\leq3$,
there are nine possibilities for
$M$ and $u$ up to isomorphism, see Figure 1. Observe that when
$|M|\leq2$ we can extend $M$ to a matching $M^{'}$ of size 3 which
satisfying $u\notin V(M^{'})$. If the conclusion holds for $M^{'}$,
then it also holds for $M$, so we can assume that $|M|=3$. Then
there is a Hamiltonian path of $Q_{3}$ between $u$ and $v$ passing
through $M$, see Figure 2.
\end{proof}

\begin{figure}[h]
\begin{center}
\includegraphics[scale=0.7]{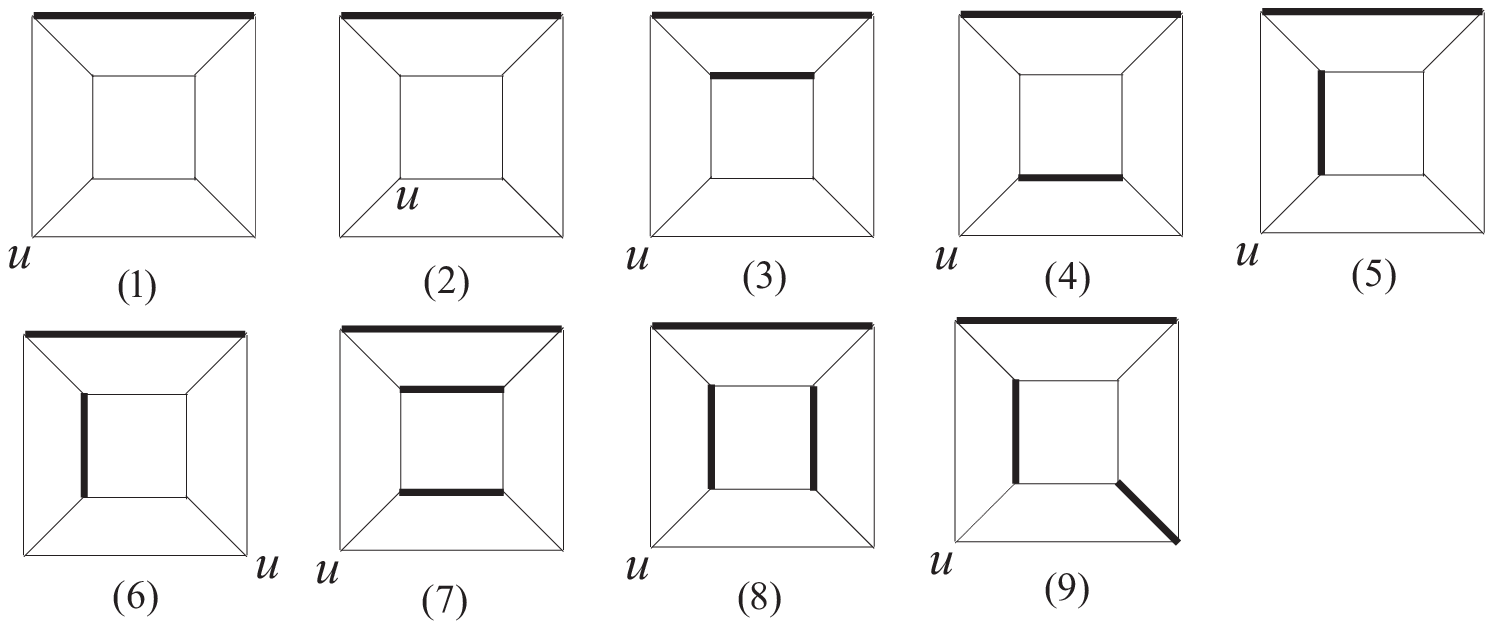}\\

{Fig. 1. Nine possibilities for $M$ and $u$ of Lemma \ref{3path}
with the edges of $M$ highlighted.}
\end{center}
\end{figure}

\begin{figure}[h]
\begin{center}
\includegraphics[scale=0.7]{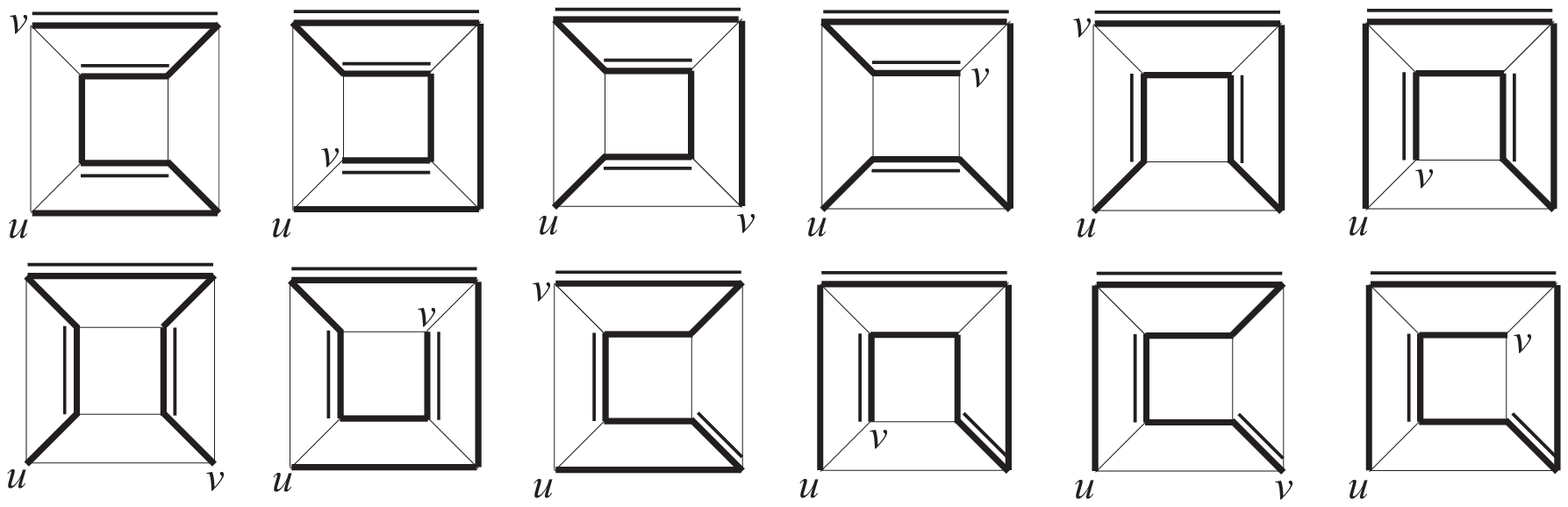}\\

{Fig. 2. Hamiltonian paths in Lemma \ref{3path} between $u$ and $v$
passing through $M$.}
\end{center}
\end{figure}

The following lemma is the base of induction of Theorem \ref{main1}.

\begin{Lemma}\label{cy} Every matching of $Q_{n}$ can be extended to
a Hamiltonian cycle of $Q_{n}$ for $n=2,3,4$.
\end{Lemma}

\begin{proof} Let $M$ be a matching of $Q_{n}$.
By Lemma \ref{Fink}, every perfect matching can be extended a
Hamiltonian cycle of $Q_{n}$ for $n\geq2$, so we only need to
consider the case that $M$ is not perfect. If $n=2,3$, then
$|M|\leq2n-3$ and therefore by Lemma \ref{forest}, there is a
Hamiltonian cycle of $Q_{n}$ containing $M$. If $n=4$, since
$|M|\leq7$, there exists $j\in[4]$ such that $|M\cap E_{j}|\leq1$.
Decompose $Q_{4}$ into $Q_{3}^{0}$ and $Q_{3}^{1}$ by $E_{j}$ such
that $M_{1}$ is not perfect. First find a Hamiltonian cycle $C_{0}$
of $Q_{3}^{0}$ containing $M_{0}$. If $M\cap E_{j}=\{uu_{1}\}$, then
select a neighbor $v$ of $u$ on $C_{0}$. Since $M$ is a matching,
$uv\notin M$. If $M\cap E_{j}=\emptyset$, then let $uv\in
E(C_{0})\setminus M_{0}$ such that $u_{1}\notin V(M_{1})$. Since
$u_{1}\notin V(M_{1})$ and $d(u_{1},v_{1})$ is odd, by Lemma
\ref{3path} there is a Hamiltonian path $P_{u_{1}v_{1}}$ of
$Q_{3}^{1}$ passing through $M_{1}$. Then the desired Hamiltonian
cycle of $Q_{4}$ is induced by edges of $(E(C_{0})\cup
E(P_{u_{1}v_{1}})\cup\{uu_{1},vv_{1}\})\setminus\{uv\}$.
\end{proof}

\section{Base of induction of the second main theorem}

\begin{Lemma}\label{iso} Let $F=\{f\}\subseteq E(Q_{3})$ and
$M$ be a matching of $Q_{3}-F$ with $|M|=2$. Then there exists a
Hamiltonian cycle containing $M$ in $Q_{3}-F$ except that
$(Q_{3},M,F)$ is the case $(b)$ or $(c)$ on Figure 3 up to
isomorphism.
\end{Lemma}

\begin{figure}[h]
\begin{center}
\includegraphics[scale=0.6]{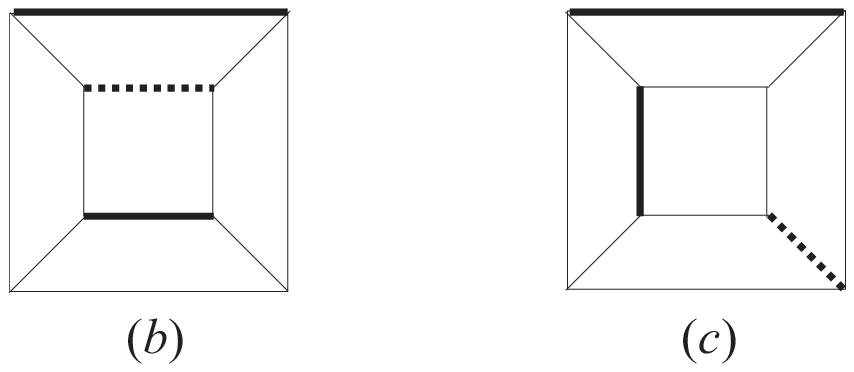}\\

{Fig. 3. Two counterexamples of Lemma \ref{iso} with the edges of
$M$ highlighted and the edge of $F$ dotted.}
\end{center}
\end{figure}

\begin{figure}[h]
\begin{center}
\includegraphics[scale=0.7]{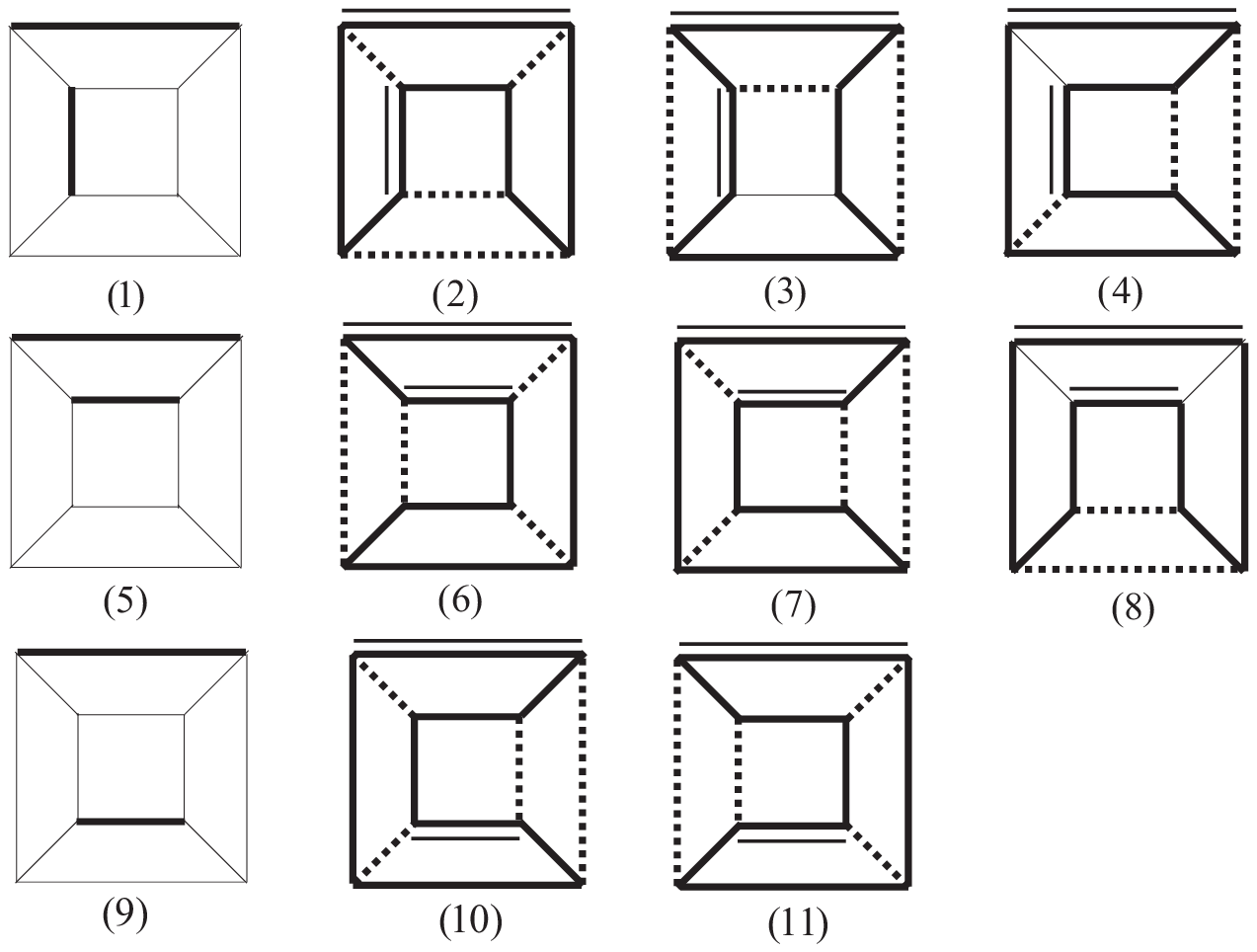}\\

{Fig. 4. Three non-isomorphic matchings of size 2 in $Q_{3}$ and
Hamiltonian cycles in Lemma \ref{iso}.}
\end{center}
\end{figure}

\begin{proof} There are three non-isomorphic matchings of size 2 in $Q_{3}$,
see Figure 4(1)(5)(9). In Figure 4(2), if $f$ is one of the dotted
edges, then there is a Hamiltonian cycle containing $M$ in
$Q_{3}-F$. By exhausting all possibilities for $\{M,f\}$, as it is
presented on Figure 4, we can verify the conclusion holds.
\end{proof}

\begin{Lemma}\label{matchingedge} Let $M$ be a matching of $Q_{3}$ with $|M|=3$.
Then there exists at most one edge $e$ in $E(Q_{3})\setminus M$ such
that $M\cup\{e\}$ is not contained in any Hamiltonian cycle of
$Q_{3}$.
\end{Lemma}

\begin{proof} There are three non-isomorphic matchings of size 3 in $Q_{3}$.
By exhausting all possibilities for $\{M,e\}$, we can find a
Hamiltonian cycle containing $M\cup\{e\}$ except that $\{M,e\}$ is
as on Figure 5(1), see Figure 5.
\end{proof}

\begin{figure}[h]
\begin{center}
\includegraphics[scale=0.7]{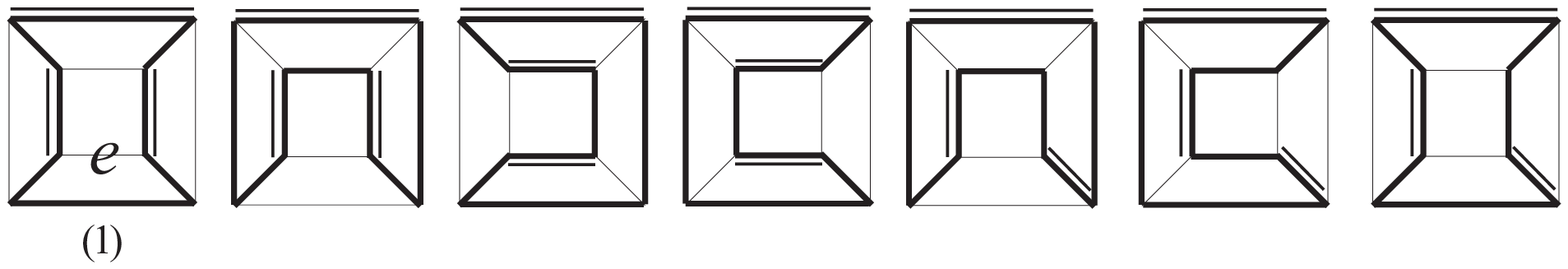}\\

{Fig. 5. Hamiltonian cycles of $Q_{3}$ containing $M\cup\{e\}$ in
Lemma \ref{matchingedge} with one exception.}
\end{center}
\end{figure}

\begin{Lemma}\label{ba1} Let $F\subseteq E(Q_{4})$ and $M$ be a matching of $Q_{4}-F$
with $|M|=4$ and $|F|=1$. If there exists $j\in[4]$ such that $M\cap
E_{j}=\emptyset$, then there exists a Hamiltonian cycle containing
$M$ in $Q_{4}-F$ except that $(Q_{4},M,F)$ is the case $(a)$ on
Figure 6 up to isomorphism.
\end{Lemma}

\begin{figure}[h]
\begin{center}
\includegraphics[scale=0.6]{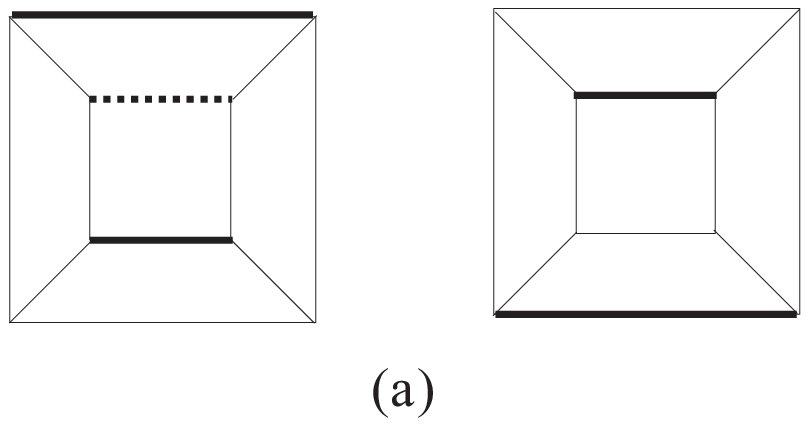}\\

{Fig. 6. The counterexample $(Q_{4},M,F)$ with the edges of $M$
highlighted and the edge of $F$ dotted, where the  edges between the
two copies of $Q_{3}$ are omitted.}
\end{center}
\end{figure}

\begin{proof} Decompose $Q_{4}$ into $Q_{3}^{0}$ and
$Q_{3}^{1}$ by $E_{j}$. By symmetry, we may assume that
$|M_{0}|\geq|M_{1}|$, moreover, we can choose $j$ such that
$|M_{0}|\leq3$. Let $F=\{f\}$.

Case 1. $|M_{0}|=3$. Let $M_{1}=\{xy\}$.

Subcase 1.1. $f\in E(Q_{3}^{0})$. Apply Lemma \ref{forest} to obtain
a Hamiltonian cycle $C_{0}$ of $Q_{3}^{0}$ containing $M_{0}$. If
$f\notin E(C_{0})$, let $uv\in E(C_{0})\setminus M_{0}$ such that
$u_{1}v_{1}\neq xy$; if $f\in E(C_{0})$ and $f_{1}\neq xy$, let
$uv=f$. By Lemma \ref{pac}, there exists a Hamiltonian path
$P_{u_{1}v_{1}}$ passing through $xy$ in $Q_{3}^{1}$. Then the
desired Hamiltonian cycle in $Q_{4}-F$ is induced by edges of
$(E(C_{0})\cup E(P_{u_{1}v_{1}})$
$\cup\{uu_{1},vv_{1}\})\setminus\{uv\}$.

If $f\in E(C_{0})$ and $f_{1}=xy$, then $f=x_{0}y_{0}$. Since
$|E(C_{0})\setminus (M_{0}\cup\{e\in E(C_{0})\mid e=f$ or $e$ is
adjacent to $f\})|\geq8-(3+3)>1$, there exists an edge $uv\in
E(C_{0})\setminus M_{0}$ such that $d(uv,f)=1$, then
$d(u_{1}v_{1},xy)=1$. By Lemma \ref{twopa}, there exist spanning
paths $P_{xy}=xy,P_{u_{1}v_{1}}$ of $Q_{3}^{1}$. Then the desired
Hamiltonian cycle in $Q_{4}-F$ is induced by edges of $(E(C_{0})\cup
E(P_{u_{1}v_{1}})$
$\cup\{x_{0}x,y_{0}y,xy,uu_{1},vv_{1}\})\setminus\{f,uv\}$.

Subcase 1.2. $f\notin E(Q_{3}^{0})$. Since $|F_{1}|\leq1$, by Lemma
\ref{cyc}, $xy$ lies on a Hamiltonian cycle $C_{1}$ in
$Q_{3}^{1}-F_{1}$. Since $|E(C_{1})|-(|M_{0}|+|M_{1}|)=4>2+1$, by
Lemma \ref{matchingedge} there exists an edge $uv\in
E(C_{1})\setminus M_{1}$ such that $u_{0}v_{0}\notin M_{0}$,
$f\notin\{u_{0}u,v_{0}v\}$ and $M_{0}\cup\{u_{0}v_{0}\}$ is
contained in some Hamiltonian cycle $C_{0}$ of $Q_{3}^{0}$. Then the
desired Hamiltonian cycle in $Q_{4}-F$ is induced by edges of
$(E(C_{0})\cup
E(C_{1})\cup\{u_{0}u,v_{0}v\})\setminus\{u_{0}v_{0},uv\}$.

Case 2. $|M_{0}|=2$

Since $|M_{0}|=|M_{1}|=2$, by symmetry, we may assume that $f\in
E_{j}\cup E(Q_{3}^{0})$. If $f\in E_{j}$, apply Lemma \ref{forest}
to find a Hamiltonian cycle $C_{0}$ of $Q_{3}^{0}$ containing
$M_{0}$. Since $|E(C_{0})\setminus(M_{0}\cup\{e\in E(C_{0})\mid e$
is adjacent to $f$ or $e_{1}\in M_{1}\})|\geq8-(2+2+2)=2$, there
exists an edge $uv\in E(C_{0})\setminus M_{0}$ such that
$u_{1}v_{1}\notin M_{1}$ and $f\notin\{uu_{1},vv_{1}\}$. If $f\in
E(Q_{3}^{0})$ and $(Q_{3}^{0},M_{0},F_{0})$ is not the case $(b)$ or
$(c)$, by Lemma \ref{iso} there exists a Hamiltonian cycle $C_{0}$
containing $M_{0}$ in $Q_{3}^{0}-F_{0}$, and let $uv\in
E(C_{0})\setminus M_{0}$ such that $u_{1}v_{1}\notin M_{1}$. If
$(Q_{3}^{0},M_{0},F_{0})$ is $(b)$ or $(c)$ and $f_{1}\notin M_{1}$,
since $M_{0}\cup\{f\}$ is a linear forest of size 3, by Lemma
\ref{forest} there is a Hamiltonian cycle $C_{0}$ of $Q_{3}^{0}$
containing $M_{0}\cup\{f\}$ and let $uv=f$, then $u_{1}v_{1}\notin
M_{1}$.

For the above three cases, since $M_{1}\cup\{u_{1}v_{1}\}$ is a
linear forest of size 3, by Lemma \ref{forest} there is a
Hamiltonian cycle $C_{1}$ containing $M_{1}\cup\{u_{1}v_{1}\}$ in
$Q_{3}^{1}$. Then the desired Hamiltonian cycle in $Q_{4}-F$ is
induced by edges of $(E(C_{0})\cup
E(C_{1})\cup\{uu_{1},vv_{1}\})\setminus\{uv,u_{1}v_{1}\}$.

If $(Q_{3}^{0},M_{0},F_{0})$ is the case $(b)$ and $f_{1}\in M_{1}$,
then let $M_{1}\setminus\{f_{1}\}=\{e\}$. Since $M$ is a matching,
we have $e\in\{e_{1},e_{2},e_{3},e_{4},e_{5},e_{6},e_{7}\}$, see
Figure 7. Then we can find a Hamiltonian cycle $C^{'}$ or $C^{''}$
containing $M$ in $Q_{4}-F$ except that $(Q_{4},M,F)$ is $(a)$.

If $(Q_{3}^{0},M_{0},F_{0})$ is the case $(c)$ and $f_{1}\in M_{1}$,
then every edge of $E(Q_{3}^{1})\setminus\{f_{1}\}$ which is not
adjacent to $f_{1}$ lies in a Hamiltonian cycle $C^{1}$ or $C^{2}$
in $Q_{4}-F$, see Figure 8. Hence $C^{1}$ or $C^{2}$ is the desired
Hamiltonian cycle in $Q_{4}-F$.
\end{proof}

\begin{figure}[h]
\begin{center}
\includegraphics[scale=0.6]{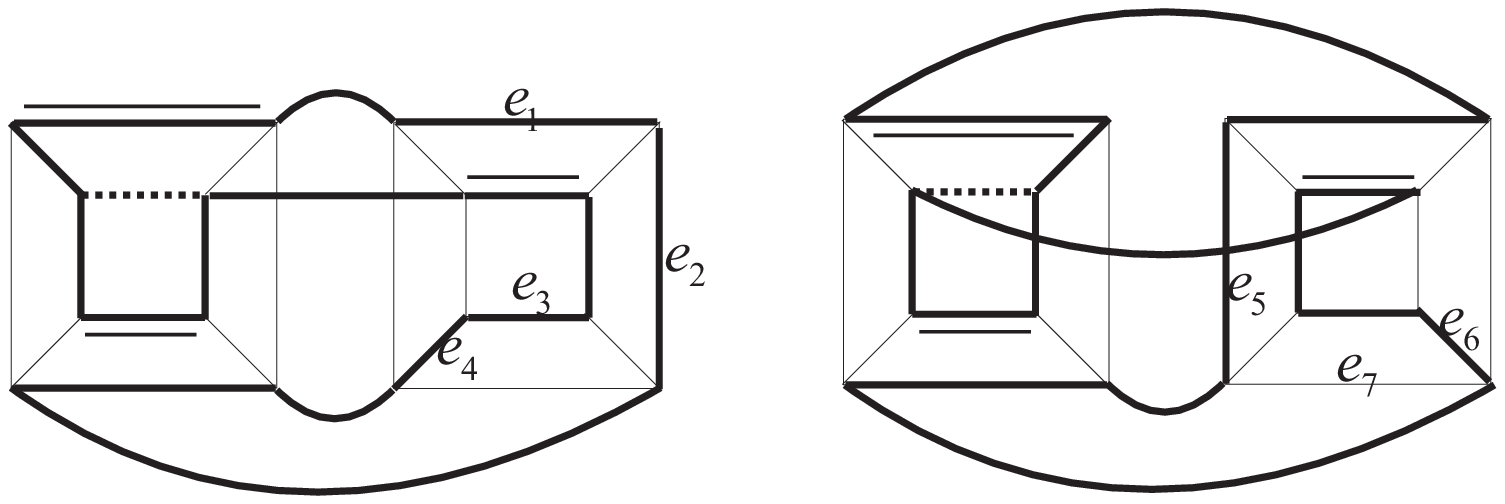}\\
{Fig. 7. The two Hamiltonian cycles $C^{'}$ and $C^{''}$ in $Q_{4}$}
\end{center}
\end{figure}

\begin{figure}[h]
\begin{center}
\includegraphics[scale=0.6]{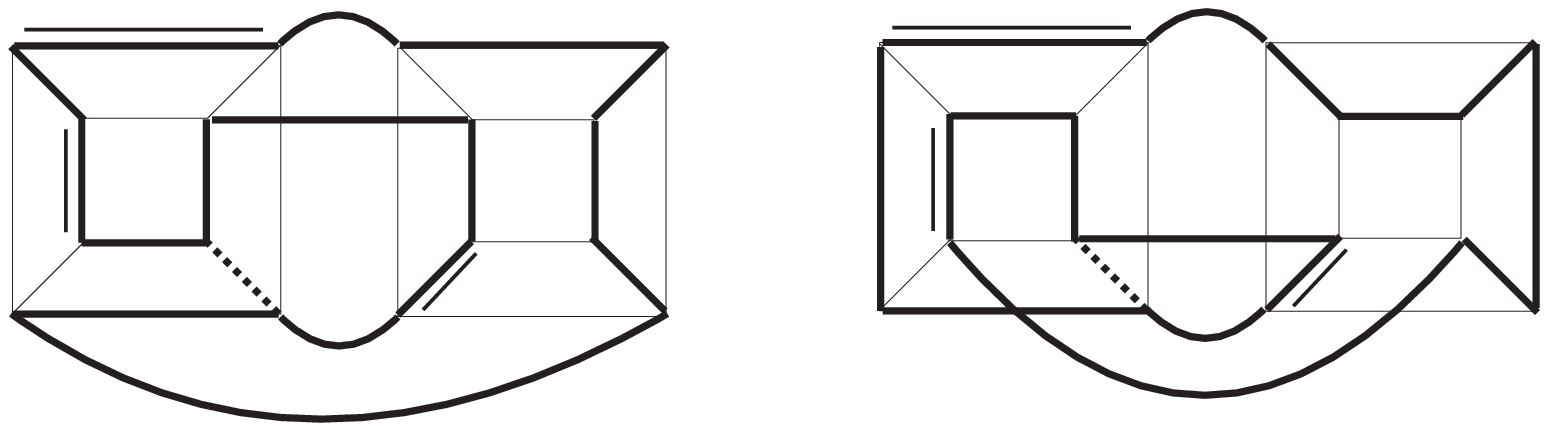}\\
{Fig. 8. The two Hamiltonian cycles $C^{1}$ and $C^{2}$ in $Q_{4}$}
\end{center}
\end{figure}

\begin{Lemma}\label{ba2} Let $F\subseteq E(Q_{4})$ and $M$ be a matching of $Q_{4}-F$
with $|M|=4$ and $|F|=1$. If $|M\cap E_{i}|=1$ for any $i\in[4]$,
then there exists a Hamiltonian cycle
containing $M$ in $Q_{4}-F$.
\end{Lemma}
\begin{proof} Since $|M\cap E_{i}|=1$ for any $i\in[4]$ and $|F|=1$, we can choose
$j\in[4]$ such that $|M\cap E_{j}|=1=|F\cap E_{j}|$. Decompose
$Q_{4}$ into $Q_{3}^{0}$ and $Q_{3}^{1}$ by $E_{j}$ such that
$|M_{0}|\geq|M_{1}|$. Let $M\cap E_{j}=\{uu_{1}\}$ and $F=F\cap
E_{j}=\{f\}$. Apply Lemma \ref{forest} to find a Hamiltonian cycle
$C_{0}$ of $Q_{3}^{0}$ containing $M_{0}$. Select a neighbor $v$ of
$u$ on $C_{0}$ such that $vv_{1}\neq f$. Then
$M_{1}\cup\{u_{1}v_{1}\}$ forms a linear forest of size at most 2
and therefore, using Lemma \ref{forest} again, there is a
Hamiltonian cycle $C_{1}$ of $Q_{3}^{1}$ containing
$M_{1}\cup\{u_{1}v_{1}\}$. Then the desired Hamiltonian cycle in
$Q_{4}-F$ is induced by edges of $(E(C_{0})\cup
E(C_{1})\cup\{uu_{1},vv_{1}\})\setminus\{uv,u_{1}v_{1}\}$.
\end{proof}

\begin{Lemma}\label{ba} Let $F\subseteq E(Q_{4})$ and $M$ be a matching of $Q_{4}-F$
with $|M|=4$ and $|F|=1$. Then there exists a Hamiltonian cycle
containing $M$ in $Q_{4}-F$ except that $(Q_{4},M,F)$ is the case
$(a)$ on Figure 6 up to isomorphism.
\end{Lemma}

In Lemma \ref{ba}, since $|M|=4$, there exists $j\in[4]$ such that
$M\cap E_{j}=\emptyset$ or $|M\cap E_{i}|=1$ for any $i\in[4]$,
so Lemma \ref{ba} holds by Lemma \ref{ba1} and \ref{ba2}.

\begin{Lemma}\label{se} Let $F=\{f,g\}\subseteq E(Q_{4})$ and
$M=\{e,h\}$ be a matching of $Q_{4}-F$.
Then there exists a Hamiltonian cycle containing $M$ in $Q_{4}-F$.
\end{Lemma}
\begin{proof} Apply Lemma \ref{decompose} to decompose $Q_{4}$ into $Q_{3}^{0}$
and $Q_{3}^{1}$ by $E_{j}$ such that $e\in E(Q_{3}^{0})$ and $h\in
E(Q_{3}^{1})$. By symmetry, we may assume that $|F_{0}|\geq|F_{1}|$,
then $|F_{1}|\leq1$.

Case 1. $F_{1}=\emptyset$. If $|F_{0}|\leq1$, by Lemma \ref{cyc},
$e$ lies on a Hamiltonian cycle $C_{0}$ in $Q_{3}^{0}-F_{0}$. Since
$|E(C_{0})|=8>2+4$, let $uv\in E(C_{0})\setminus\{e\}$ such that
$\{uu_{1},vv_{1}\}\cap F=\emptyset$ and $u_{1}v_{1}\neq h$. If
$F_{0}=F$, since at least one of $f_{1}$ and $g_{1}$ is not $h$, we
may assume that $f_{1}\neq h$. By Lemma \ref{cyc}, $e$ lies on a
Hamiltonian cycle $C_{0}$ in $Q_{3}^{0}-\{g\}$. If $f\notin
E(C_{0})$, let $uv\in E(C_{0})\setminus\{e\}$ such that
$u_{1}v_{1}\neq h$; if $f\in E(C_{0})$, let $uv=f$, note that
$u_{1}v_{1}=f_{1}\neq h$.

Next apply Lemma \ref{pac} to find a Hamiltonian path
$P_{u_{1}v_{1}}$ of $Q_{3}^{1}$ passing through $h$. Then the
desired Hamiltonian cycle in $Q_{4}-F$ is induced by edges of
$(E(C_{0})\cup
E(P_{u_{1}v_{1}})\cup\{uu_{1},vv_{1}\})\setminus\{uv\}$.

Case 2. $|F_{1}|=1$. Then $|F_{0}|=|F_{1}|=1$ and $F\cap
E_{j}=\emptyset$. By Lemma \ref{cyc}, $e$ lies on a Hamiltonian
cycle $C_{0}$ in $Q_{3}^{0}-F_{0}$ and $h$ lies on a Hamiltonian
cycle $C_{1}$ in $Q_{3}^{1}-F_{1}$. Since
$|E(C_{0})\setminus\{e\}|-|E(Q_{3}^{1})\setminus E(C_{1})|=7-4>1$,
there has to be an edge $uv\in E(C_{0})\setminus\{e\}$ such that
$u_{1}v_{1}\in E(C_{1})\setminus\{h\}$. Then the desired Hamiltonian
cycle in $Q_{4}-F$ is induced by edges of $(E(C_{0})\cup
E(C_{1})\cup\{uu_{1},vv_{1}\})\setminus\{uv,u_{1}v_{1}\}$.
\end{proof}

Using Lemmas \ref{cycle}, \ref{cy}, \ref{ba} and \ref{se}, we
obtain,

\begin{Lemma}\label{base} Let $F\subseteq E(Q_{4})$ and $M$ be a matching
of $Q_{4}-F$ with $1\leq|M|\leq6$ and $|F|\leq3-\lceil
\frac{|M|}{2}\rceil$. Then there exists a Hamiltonian cycle
containing $M$ in $Q_{4}-F$ except that $(Q_{4},M,F)$ is the case
$(a)$ on Figure 6 up to isomorphism.
\end{Lemma}

\begin{Lemma}\label{decom} Let $F\subseteq E(Q_{5})$ and $M$ be a matching
of $Q_{5}-F$ with $1\leq|M|\leq8$ and
$|F|\leq4-\lceil\frac{|M|}{2}\rceil$. Then there exists $j\in[5]$
such that $|E_{j}\cap(M\cup F)|\leq1$ and none of
$(Q_{4}^{0},M_{0},F_{0})$ and $(Q_{4}^{1},M_{1},F_{1})$ is the case
$(a)$ on Figure 6 up to isomorphism, where $Q_{4}^{0}\cup
Q_{4}^{1}=Q_{5}-E_{j}$.
\end{Lemma}
\begin{proof} If $F=\emptyset$ or $|M|\leq3$, since $|M\cup
F|\leq|M|+4-\lceil\frac{|M|}{2}\rceil\leq8$, we can select $j\in[5]$
such that $|E_{j}\cap(M\cup F)|\leq1$. Moreover, none of
$(Q_{4}^{0},M_{0},F_{0})$ and $(Q_{4}^{1},M_{1},F_{1})$ is $(a)$. If
$|M|=4$, then $|F|\leq2$ and $|M\cup F|\leq6$. First select
$j\in[5]$ such that $|E_{j}\cap(M\cup F)|\leq1$. If one of
$(Q_{4}^{0},M_{0},F_{0})$ and $(Q_{4}^{1},M_{1},F_{1})$ is the case
$(a)$, then we can choose $j_{0}\in[5]$ such that $|F\cap
E_{j_{0}}|\leq1, M\cap E_{j_{0}}=\emptyset$ and the two subcubes
decomposed by $E_{j_{0}}$ both contain two edges of $M$. Hence $j$
or $j_{0}$ satisfies the lemma. So it suffices to prove the case
$5\leq|M|\leq6$ and $|F|=1$.

Since $|M|+|F|\leq7$, there exists $i\in[5]$ such that
$|E_{i}\cap(M\cup F)|\leq1$. Decompose $Q_{5}$ into two subcubes
$Q_{4}^{0_{i}}$ and $Q_{4}^{1_{i}}$ by $E_{i}$. If none of
$(Q_{4}^{0_{i}},M_{0},F_{0})$ and $(Q_{4}^{1_{i}},M_{1},F_{1})$ is
the case $(a)$, then the conclusion holds. Otherwise, assume that
$(Q_{4}^{0_{i}},M_{0},F_{0})$ is $(a)$. Without loss of generality,
assume that $i=1$ and $(M_{0}\cup F) \subseteq E_{2}$, then $|(M\cup
F)\cap(E_{3}\cup E_{4}\cup E_{5})|\leq2$. Hence, at least two of
$E_{3},E_{4}$ and $E_{5}$ satisfy $|E_{j}\cap(M\cup F)|\leq1$, where
$j\in\{3,4,5\}$. Assume that $|E_{4}\cap(M\cup F)|\leq1$ and
$|E_{5}\cap(M\cup F)|\leq1$. We claim that $j=4$ or $j=5$ satisfies
the lemma. If $j=4$ does not satisfy the lemma, then one of the two
subcubes decomposed from $Q_{5}$ by $E_{4}$ is the case $(a)$, see
Figure 9. Then the two subcubes decomposed from $Q_{5}$ by $E_{5}$
both contain three edges of $M$, so none of them is the case $(a)$,
hence $j=5$ satisfies the lemma. The proof is complete.
\end{proof}

\begin{figure}[h]
\begin{center}
\includegraphics[scale=0.7]{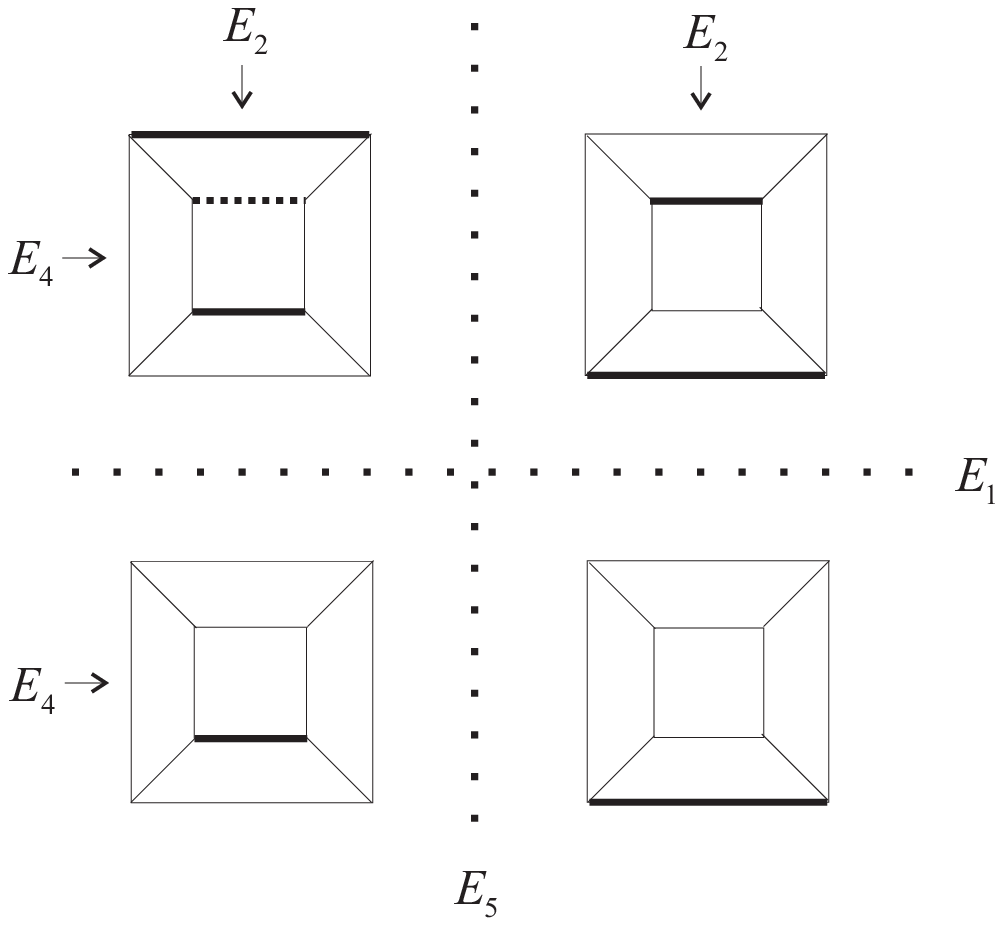}\\

{Fig. 9. $(Q_{5},M,F)$ satisfying one of the two subcubes decomposed
by $E_{i}$ is the case $(a)$ where $i=1,4$.}
\end{center}
\end{figure}

\section{Main results}

\begin{thm}\label{main1} For $n\geq2$, let $M$ be a matching
of $Q_{n}$ with $|M|\leq2n-1$. Then there exists a Hamiltonian cycle
of $Q_{n}$ containing $M$.
\end{thm}
\begin{proof} We prove the theorem by induction on $n$. The theorem holds for
$n=2,3,4$ by Lemma \ref{cy}. Assume that the theorem holds for
$n-1(\geq4)$, we are to show it holds for $n(\geq5)$. Since
$|M|\leq2n-1$, there exists $j\in[n]$ such that $|M\cap
E_{j}|\leq1$. Decompose $Q_{n}$ into $Q_{n-1}^{0}$ and $Q_{n-1}^{1}$
by $E_{j}$ such that $|M_{0}|\geq|M_{1}|$, then $|M_{1}|\leq n-1$.

\textbf{Claim 1.} If there exists a Hamiltonian cycle $C_{0}$
containing $M_{0}$ in $Q_{n-1}^{0}$, then we can construct a
Hamiltonian cycle of $Q_{n}$ containing $M$.

If $M\cap E_{j}=\{uu_{1}\}$, select a neighbor $v$ of $u$ on
$C_{0}$. Since $M$ is a matching, we have $\{uv,u_{1}v_{1}\}\cap
M=\emptyset$. If $M\cap E_{j}=\emptyset$, since
$|E(C_{0})|-(|M_{0}|+|M_{1}|)\geq2^{n-1}-(2n-1)\geq1$ for $n\geq5$,
there exists an edge $uv\in E(C_{0})\setminus M_{0}$ such that
$u_{1}v_{1}\notin M_{1}$. Since $M_{1}\cup\{u_{1}v_{1}\}$ is a
linear forest with $|M_{1}\cup\{u_{1}v_{1}\}|\leq
n-1+1\leq2(n-1)-3$, by Lemma \ref{forest}, there exists a
Hamiltonian cycle $C_{1}$ containing $M_{1}\cup\{u_{1}v_{1}\}$ in
$Q_{n-1}^{1}$. Then the desired Hamiltonian cycle of $Q_{n}$ is
induced by edges of $(E(C_{0})\cup
E(C_{1})\cup\{uu_{1},vv_{1}\})\setminus\{uv,u_{1}v_{1}\}$. Thus
Claim 1 is proved.

If $|M_{0}|\leq2(n-1)-1$, by the induction hypothesis, there exists
a Hamiltonian cycle $C_{0}$ containing $M_{0}$ in $Q_{n-1}^{0}$ and
therefore, the theorem holds by Claim 1. So in what follows we can
assume that $2n-2\leq|M_{0}|\leq|M|\leq2n-1$. We distinguish two
cases to consider.

Case 1. $|M_{0}|=2n-2$

Let $xy\in M_{0}$ and apply the induction hypothesis to find a
Hamiltonian cycle $C_{0}$ containing $M_{0}\setminus\{xy\}$ in
$Q_{n-1}^{0}$. If $xy\in E(C_{0})$, then $C_{0}$ is a Hamiltonian
cycle containing $M_{0}$ in $Q_{n-1}^{0}$ and therefore, the theorem
holds by Claim 1. If $xy\notin E(C_{0})$, choose neighbors $s$ and
$t$ of $x$ and $y$ on $C_{0}$ such that one of the paths between $x$
and $y$ contains $s$ and the other contains $t$. Note that there are
two ways to choose $s$ and $t$.

If $M\cap E_{j}=\emptyset$, since $2n-2=|M_{0}|\leq|M|\leq2n-1$, we
have $|M_{1}|\leq1$ and therefore, moreover, we can choose $s$ and
$t$ such that $s_{1}t_{1}\notin M_{1}$. Since $M$ is a matching and
$xy\in M$, we have $\{xs,yt\}\cap M=\emptyset$. Since
$d(s_{1},t_{1})$ is odd, by Lemma \ref{Havel} and \ref{pac}, there
exists a Hamiltonian path $P_{s_{1}t_{1}}$ of $Q_{n-1}^{1}$ passing
through $M_{1}$. Then the desired Hamiltonian cycle of $Q_{n}$ is
induced by edges of $(E(C_{0})\cup
E(P_{s_{1}t_{1}})\cup\{xy,ss_{1},tt_{1}\})\setminus\{xs,yt\})$.

If $M\cap E_{j}=\{uu_{1}\}$, then $M_{1}=\emptyset$. Since $Q_{n}$
is bipartite and $d(x,y)$ is odd, we can choose $s$ and $t$ such
that $u\notin\{s,t\}$. Select a neighbor $v$ of $u$ on $C_{0}$ such
that $v\notin\{s,t\}$. Since $d(u,v)$ and $d(s,t)$ are odd,
$d(u_{1},v_{1})$ and $d(s_{1},t_{1})$ are also odd, by Lemma
\ref{twopa}, there exist spanning paths
$P_{u_{1}v_{1}},P_{s_{1}t_{1}}$ of $Q_{n-1}^{1}$. Hence the desired
Hamiltonian cycle of $Q_{n}$ is induced by edges of $(E(C_{0})\cup
E(P_{u_{1}v_{1}})\cup E(P_{s_{1}t_{1}})$
$\cup\{xy,uu_{1},vv_{1},ss_{1},tt_{1}\})\setminus\{uv,xs,yt\}$.

Case 2. $|M_{0}|=2n-1$

Since $2n-1=|M_{0}|\leq|M|\leq2n-1$, we have $M=M_{0}$ and $M\cap
E_{j}=\emptyset=M_{1}$. Let $xy,uv\in M_{0}$ and apply the induction
hypothesis to find a Hamiltonian cycle $C_{0}$ containing
$M_{0}\setminus\{xy,uv\}$ in $Q_{n-1}^{0}$. If $\{xy,uv\}\subseteq
E(C_{0})$, then the conclusion holds by Claim 1. If $|\{xy,uv\}\cap
E(C_{0})|=1$, without loss of generality, assume that $xy\notin
E(C_{0})$ and $uv\in E(C_{0})$. Choose neighbors $s$ and $t$ of $x$
and $y$ on $C_{0}$ such that one of the paths between $x$ and $y$
contains $s$ and the other contains $t$. Since $d(s_{1},t_{1})$ is
odd, by Lemma \ref{Havel}, there is a Hamiltonian path
$P_{s_{1}t_{1}}$ of $Q_{n-1}^{1}$. Then the desired Hamiltonian
cycle of $Q_{n}$ is induced by edges of $(E(C_{0})\cup
E(P_{s_{1}t_{1}})\cup\{xy,ss_{1},tt_{1}\})\setminus\{xs,yt\})$.

If $\{xy,uv\}\cap E(C_{0})=\emptyset$, there are two cases up to
isomorphism, see Figure 10. Choose neighbors $r,s,w,t$ of $x,y,u,v$
on $C_{0}$, respectively, see Figure 10. Since $x,y,u,v$ are
pairwise distinct and $\{xy,uv\}\subseteq M$, we have $r,s,w,t$ are
pairwise distinct and $\{xr,ys,uw,vt\}\cap M=\emptyset$.

\begin{figure}[h]
\begin{center}
\includegraphics[scale=0.5]{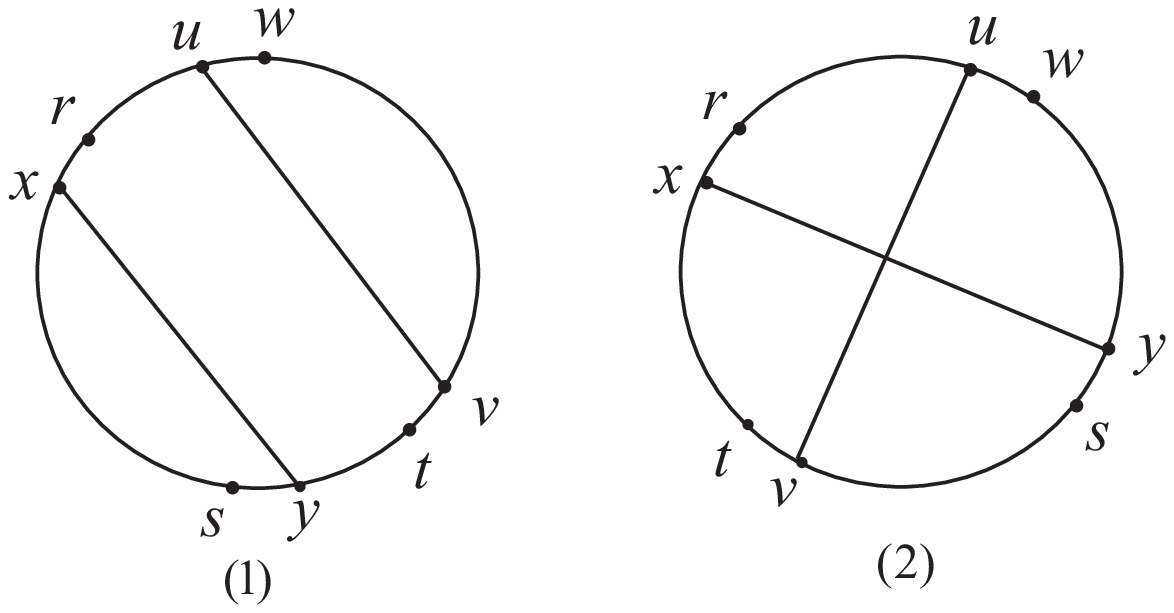}\\
{Fig. 10. The choice of neighbors $r,s,w,t$ of $x,y,u,v$ on $C_{0}$}
\end{center}
\end{figure}

If one of the paths on $C_{0}$ between $u$ and $v$ contains both $x$
and $y$, see Figure 10(1). Since $d(r,s)$ and $d(w,t)$ are odd,
$d(r_{1},s_{1})$ and $d(w_{1},t_{1})$ are also odd, by Lemma
\ref{twopa}, there exist spanning paths
$P_{r_{1}s_{1}},P_{w_{1}t_{1}}$ of $Q_{n-1}^{1}$. Hence the desired
Hamiltonian cycle of $Q_{n}$ is induced by edges of $(E(C_{0})\cup
E(P_{r_{1}s_{1}})\cup E(P_{w_{1}t_{1}})$
$\cup\{uv,xy,rr_{1},ss_{1},ww_{1},tt_{1}\})\setminus\{xr,ys,uw,vt\}$.

If one of the paths on $C_{0}$ between $u$ and $v$ contains $x$ and
the other contains $y$, see Figure 10(2). Since $Q_{n}$ is bipartite
and $d(u,v)$ is odd, $d(x,u)+d(x,v)$ is odd. Without loss of
generality, assume that $d(x,v)$ is odd, then $d(y,u)$ is odd and
therefore, $d(r_{1},t_{1})$ and $d(w_{1},s_{1})$ are both odd. By
Lemma \ref{twopa}, there exist spanning paths
$P_{r_{1}t_{1}},P_{w_{1}s_{1}}$ of $Q_{n-1}^{1}$. Hence the desired
Hamiltonian cycle of $Q_{n}$ is induced by edges of $(E(C_{0})\cup
E(P_{r_{1}t_{1}})\cup E(P_{w_{1}s_{1}})$
$\cup\{uv,xy,rr_{1},ss_{1},ww_{1},tt_{1}\})\setminus\{xr,ys,uw,vt\}$.
\end{proof}

\begin{thm}\label{main} For $n\geq4$, let $F\subseteq E(Q_{n})$ and $M$
be a matching of $Q_{n}-F$ with $1\leq|M|\leq2n-2$ and $|F|\leq
n-1-\lceil\frac{|M|}{2}\rceil$. Then there exists a Hamiltonian
cycle containing $M$ in $Q_{n}-F$ except that $(Q_{n},M,F)$ is the
case $(a)$ depicted on Figure 6 up to isomorphism.
\end{thm}
\begin{proof} By Theorem \ref{main1}, the theorem is true if $F=\emptyset$.
Hence we only consider the case of $|F|\geq1$, this implies that
$|M|\leq2n-4$. When $|M|=1$, since $|F|\leq n-2$, the theorem is
true by Lemma \ref{cyc}. Moreover, we observe that, if
$|F|<n-1-\lceil\frac{|M|}{2}\rceil$, we may arbitrarily delete some
more edges (not in $M$) to make a faulty edge set $F^{'}\supseteq F$
and $|F^{'}|=n-1-\lceil\frac{|M|}{2}\rceil$. If our result holds for
$F^{'}$, it holds for $F$. So it suffices to consider $2\leq|M|\leq
2n-4$ and $|F|=n-1-\lceil\frac{|M|}{2}\rceil$.

We prove Theorem \ref{main} by induction on $n$. By Lemma
\ref{base}, Theorem \ref{main} holds for $n=4$. Suppose Theorem
\ref{main} holds for $n-1(\geq4)$.

Since $|M|+|F|\leq2n-3$, there exists $j\in[n]$ such that $|(M\cup
F)\cap E_{j}|\leq1$. Suppose $Q_{n}$ is decomposed into two
$(n-1)$-dimensional subcubes $Q_{n-1}^{0}$ and $Q_{n-1}^{1}$ by
$E_{j}$. Moreover, in the case when $n=5$, the choice can be made in
such a way that none of $(Q_{4}^{0},M_{0},F_{0})$ and
$(Q_{4}^{1},M_{1},F_{1})$ is the case $(a)$ on Figure 6 by Lemma
\ref{decom}. By symmetry, we may assume that $|M_{0}|\geq|M_{1}|$,
and when $|M_{0}|=|M_{1}|$, suppose $|F_{0}|\geq|F_{1}|$. Since
$2\leq|M|\leq2n-4$, we have $1\leq|M_{0}|\leq2(n-1)-2$.

Case 1. $(M\cup F)\cap E_{j}=\emptyset$

Subcase 1.1. $|M_{0}|\leq|M|-2$. Note that then $2\leq|M_{1}|\leq
n-2$.

Since $|F_{0}|\leq |F|\leq(n-1)-1-\lceil\frac{|M_{0}|}{2}\rceil$, by
the induction hypothesis, there exists a Hamiltonian cycle $C_{0}$
containing $M_{0}$ in $Q_{n-1}^{0}-F_{0}$. Since
$|E(C_{0})|-(|M_{0}|+|M_{1}|)\geq2^{n-1}-(2n-4)\geq1$, there exists
an edge $uv\in E(C_{0})\setminus M_{0}$ such that $u_{1}v_{1}\notin
M_{1}$, then $M_{1}\cup\{u_{1}v_{1}\}$ is a linear forest with
$1\leq|M_{1}\cup \{u_{1}v_{1}\}|\leq n-2+1\leq2(n-1)-3$. If
$|M_{0}|\geq|M_{1}|+2$, then $|M_{0}|\geq4$, so $|F_{1}|\leq
|F|=n-1-\lceil\frac{|M_{0}|+|M_{1}|}{2}\rceil\leq(n-1)-2
-\lceil\frac{|M_{1}|}{2}\rceil=
(n-1)-2-\lfloor\frac{|M_{1}|+1}{2}\rfloor$. Similarly, when
$|M_{0}|=|M_{1}|$ or $|M_{0}|=|M_{1}|+1$, we can also obtain that
$|F_{1}|\leq(n-1)-2-\lfloor\frac{|M_{1}|+1}{2}\rfloor$. By Lemma
\ref{cycle}, there exists a Hamiltonian cycle $C_{1}$ containing
$M_{1}\cup\{u_{1}v_{1}\}$ in $Q_{n-1}^{1}-(F_{1}\setminus
\{u_{1}v_{1}\})$. Then the desired Hamiltonian cycle in $Q_{n}-F$ is
induced by edges of $(E(C_{0})\cup
E(C_{1})\cup\{uu_{1},vv_{1}\})\setminus\{uv,u_{1}v_{1}\}$.

Subcase 1.2. $|M_{0}|=|M|-1$. Then $|M_{0}|\leq2(n-1)-3$. Let
$M_{1}=\{wt\}$. We distinguish the following two possibilities.

Subcase 1.2.1. $F_{1}=\emptyset$. Let $xy\in M_{0}$. Since
$|F_{0}|\leq(n-1)-1- \lceil\frac{|M_{0}\setminus\{xy\}|}{2}\rceil$,
by the induction hypothesis, there exists a Hamiltonian cycle
$C_{0}$ containing $M_{0}\setminus\{xy\}$ in $Q_{n-1}^{0}-F_{0}$. If
$xy\in E(C_{0})$, let $uv\in E(C_{0})\setminus M_{0}$ such that
$u_{1}v_{1}\neq wt$, by Lemma \ref{pac}, there is a Hamiltonian path
$P_{u_{1}v_{1}}$ passing through edge $wt$ in $Q_{n-1}^{1}$. Then
the desired Hamiltonian cycle in $Q_{n}-F$ is induced by edges of
$(E(C_{0})\cup
E(P_{u_{1}v_{1}})\cup\{uu_{1},vv_{1}\})\setminus\{uv\}$. If
$xy\notin E(C_{0})$, let $u$ and $v$ be neighbors of $x$ and $y$ on
$C_{0}$ such that one of the paths between $x$ and $y$ contains $u$
and the other contains $v$ and $u_{1}v_{1}\neq wt$. By Lemma
\ref{pac}, there is a Hamiltonian path $P_{u_{1}v_{1}}$ passing
through edge $wt$ in $Q_{n-1}^{1}$. Then the desired Hamiltonian
cycle in $Q_{n}-F$ is induced by edges of $(E(C_{0})\cup
E(P_{u_{1}v_{1}})\cup\{xy,uu_{1},vv_{1}\})\setminus\{xu,yv\}$.

Subcase 1.2.2. $F_{1}\neq\emptyset$. Since $|F_{0}|\leq|F|-1\leq
(n-1)-1-\lceil\frac{|M_{0}|}{2}\rceil$, by the induction hypothesis,
there exists a Hamiltonian cycle $C_{0}$ containing $M_{0}$ in
$Q_{n-1}^{0}-F_{0}$. Since $|E(C_{0})\setminus(M_{0}\cup\{e\in
E(C_{0})\mid e_{1}$ is incident with $w$ or $t\})|\geq
2^{n-1}-(2n-5+4)\geq1$, there exists an edge $uv\in
E(C_{0})\setminus M_{0}$ such that $u_{1}v_{1}\neq wt$ and $\{wt,
u_{1}v_{1}\}$ is a matching of $Q_{n-1}^{1}$. Similar to Subcase
1.1, we can check that $|F_{1}|\leq n-2-\lceil\frac{2}{2}\rceil$. By
the induction hypothesis, there exists a Hamiltonian cycle $C_{1}$
containing $\{wt,u_{1}v_{1}\}$ in
$Q_{n-1}^{1}-(F_{1}\setminus\{u_{1}v_{1}\})$. Hence the desired
Hamiltonian cycle in $Q_{n}-F$ is induced by edges of $(E(C_{0})\cup
E(C_{1})\cup\{uu_{1},vv_{1}\})\setminus\{uv,u_{1}v_{1}\}$.

Subcase 1.3. $M_{0}=M$. Recall that $|M_{0}|=|M|\leq2(n-1)-2$.

If $F_{0}\neq\emptyset$, let $f\in F_{0}$, since $|F_{0}\setminus
\{f\}|\leq n-2-\lceil\frac{|M_{0}|}{2}\rceil$, by the induction
hypothesis, there exists a Hamiltonian cycle $C_{0}$ containing
$M_{0}$ in $Q_{n-1}^{0}-(F_{0}\setminus\{f\})$. If $f\notin
E(C_{0})$, let $uv$ be any edge in $E(C_{0})\setminus M_{0}$; if
$f\in E(C_{0})$, let $uv=f$. If $F_{0}=\emptyset$ and $|M|\geq3$,
since $|M_{0}|\leq2(n-1)-2$, by the induction hypothesis, there
exists a Hamiltonian cycle $C_{0}$ containing $M_{0}$ in
$Q_{n-1}^{0}$. Let $uv$ be any edge in $E(C_{0})\setminus M_{0}$.

For the above two cases, we can verify that $|F_{1}|\leq(n-1)-2$ and
therefore, by Lemma \ref{cyc}, $u_{1}v_{1}$ lies on a Hamiltonian
cycle $C_{1}$ in $Q_{n-1}^{1}-(F_{1}\setminus\{u_{1}v_{1}\})$. Then
the desired Hamiltonian cycle in $Q_{n}-F$ is induced by edges of
$(E(C_{0})\cup
E(C_{1})\cup\{uu_{1},vv_{1}\})\setminus\{uv,u_{1}v_{1}\}$.

If $F_{0}=\emptyset$ and $|M|=2$, since $|F_{1}|=n-2\geq3$, let
$uv\in F_{1}$ such that $u_{0}v_{0}\notin M$. Since
$|F_{1}\setminus\{uv\}|=(n-1)-2$, by Lemma \ref{cyc}, $uv$ lies on a
Hamiltonian cycle $C_{1}$ in $Q_{n-1}^{1}-(F_{1}\setminus\{uv\})$.
Since $M_{0}\cup\{u_{0}v_{0}\}$ is a linear forest of size 3, by
Lemma \ref{forest}, there exists a Hamiltonian cycle $C_{0}$
containing $M_{0}\cup\{u_{0}v_{0}\}$ in $Q_{n-1}^{0}$. Then the
desired Hamiltonian cycle in $Q_{n}-F$ is induced by edges of
$(E(C_{0})\cup
E(C_{1})\cup\{u_{0}u,v_{0}v\})\setminus\{u_{0}v_{0},uv\}$.

Case 2. $E_{j}\cap M=\emptyset$ and $E_{j}\cap F=\{xx_{1}\}$

Since $|F_{0}|\leq n-2-\lceil\frac{|M_{0}|}{2}\rceil$, by the
induction hypothesis, there exists a Hamiltonian cycle $C_{0}$
containing $M_{0}$ in $Q_{n-1}^{0}-F_{0}$. Since $|E(C_{0})\setminus
(M_{0}\cup\{e\in E(C_{0})\mid e_{1}\in M_{1}$ or $e$ is incident
with $x\})|\geq2^{n-1}-(2n-4+2)\geq1$, there exists an edge $uv\in
E(C_{0})\setminus M_{0}$ such that $xx_{1}\notin\{uu_{1},vv_{1}\}$
and $u_{1}v_{1}\notin M_{1}$, then $M_{1}\cup\{u_{1}v_{1}\}$ is a
linear forest with $1\leq|M_{1}\cup\{u_{1}v_{1}\}|\leq2(n-1)-3$. If
$|M_{0}|\geq2$, then $|F_{1}|\leq
n-2-(\lceil\frac{|M_{1}|}{2}\rceil+1)=
(n-1)-2-\lfloor\frac{|M_{1}|+1}{2}\rfloor$. Similarly, when
$|M_{0}|=1$, $|F_{1}|\leq(n-1)-2-\lfloor\frac{|M_{1}|+1}{2}\rfloor$.
By Lemma \ref{cycle}, there exists a Hamiltonian cycle $C_{1}$
containing $M_{1}\cup\{u_{1}v_{1}\}$ in $Q_{n-1}^{1}-(F_{1}\setminus
\{u_{1}v_{1}\})$. Then the desired Hamiltonian cycle in $Q_{n}-F$ is
induced by edges of $(E(C_{0})\cup E(C_{1})\cup
\{uu_{1},vv_{1}\})\setminus\{uv,u_{1}v_{1}\}$.

Case 3. $E_{j}\cap F=\emptyset$ and $E_{j}\cap M=\{uu_{1}\}$

When $|M|=2$ or $3$, since $|M\cup F|=n-1+\lfloor
\frac{|M|}{2}\rfloor=n$, we can select $j_{0}\in[n]$ such that
$E_{j_{0}}\cap(M\cup F)=\emptyset$ or $E_{j_{0}}\cap M=\emptyset,
|E_{j_{0}}\cap F|=1$, and therefore these two situations can be
reduced to Case 1 or Case 2. So in what follows we can suppose that
$|M|\geq4$.

Subcase 3.1. $|M_{0}|=|M|-1$. We distinguish the following two possibilities.

Subcase 3.1.1. $F_{1}=\emptyset$. Let $f\in F_{0}$. Since
$|F_{0}\setminus\{f\}|\leq n-2-\lceil\frac{|M_{0}|}{2}\rceil$, by
the induction hypothesis, there exists a Hamiltonian cycle $C_{0}$
containing $M_{0}$ in $Q_{n-1}^{0}-(F_{0}\setminus\{f\})$. If
$f\notin E(C_{0})$, let $uv\in E(C_{0})\setminus M_{0}$; if $f\in
E(C_{0})$ and $u$ is incident with $f$, let $uv=f$. By Lemma
\ref{Havel}, there exists a Hamiltonian path $P_{u_{1}v_{1}}$ in
$Q_{n-1}^{1}$. Then the desired Hamiltonian cycle in $Q_{n}-F$ is
induced by edges of $(E(C_{0})\cup
E(P_{u_{1}v_{1}})\cup\{uu_{1},vv_{1}\})\setminus\{uv\}$. If $f\in
E(C_{0})$ and $u$ is not incident with $f$, let $f=xy$ and $uv\in
E(C_{0})\setminus M_{0}$ such that $v\notin\{x,y\}$. Since
$d(u_{1},v_{1})=d(x_{1},y_{1})=1$, by Lemma \ref{twopa}, there exist
spanning paths $P_{u_{1}v_{1}},P_{x_{1}y_{1}}$ of $Q_{n-1}^{1}$.
Then the desired Hamiltonian cycle in $Q_{n}-F$ is induced by edges
of $(E(C_{0})\cup E(P_{u_{1}v_{1}})\cup E(P_{x_{1}y_{1}})$
$\cup\{uu_{1},vv_{1},xx_{1},yy_{1}\})\setminus\{uv,f\}$.

Subcase 3.1.2. $F_{1}\neq\emptyset$. Since $|F_{0}|\leq
n-2-\lceil\frac{|M_{0}|}{2}\rceil$, by the induction hypothesis,
there exists a Hamiltonian cycle $C_{0}$ containing $M_{0}$ in
$Q_{n-1}^{0}-F_{0}$. Let $uv\in E(C_{0})\setminus M_{0}$. Since
$|M|\geq4$, we have $|F_{1}|\leq(n-1)-2$. By Lemma \ref{cyc},
$u_{1}v_{1}$ lies on a Hamiltonian cycle $C_{1}$ in
$Q_{n-1}^{1}-(F_{1}\setminus\{u_{1}v_{1}\})$. Then the desired
Hamiltonian cycle in $Q_{n}-F$ is induced by edges of $(E(C_{0})\cup
E(C_{1})\cup\{uu_{1},vv_{1}\})\setminus\{uv,u_{1}v_{1}\}$.

Subcase 3.2. $|M_{0}|\leq|M|-2$. Note that $|M_{0}|\geq2$.

If $F_{0}\neq\emptyset$, since $|F_{0}|\leq
n-2-\lceil\frac{|M_{0}|}{2}\rceil$, by the induction hypothesis,
there exists a Hamiltonian cycle $C_{0}$ containing $M_{0}$ in
$Q_{n-1}^{0}-F_{0}$. Let $v$ be a neighbor of $u$ on $C_{0}$, then
$M_{1}\cup\{u_{1}v_{1}\}$ is a linear forest with $1\leq|M_{1}\cup
\{u_{1}v_{1}\}|\leq2(n-1)-3$. Since $|F_{1}|\leq
|F|-1\leq(n-1)-2-\lfloor\frac{|M_{1}|+1}{2}\rfloor$, by Lemma
\ref{cycle}, there exists a Hamiltonian cycle $C_{1}$ containing
$M_{1}\cup\{u_{1}v_{1}\}$ in $Q_{n-1}^{1}-(F_{1}\setminus
\{u_{1}v_{1}\})$. Then the desired Hamiltonian cycle in $Q_{n}-F$ is
induced by edges of $(E(C_{0})\cup
E(C_{1})\cup\{uu_{1},vv_{1}\})\setminus\{uv,u_{1}v_{1}\}$.

If $F_{0}=\emptyset$, since $|F_{1}|\leq
n-2-\lceil\frac{|M_{1}|}{2}\rceil$, by the induction hypothesis,
there exists a Hamiltonian cycle $C_{1}$ containing $M_{1}$ in
$Q_{n-1}^{1}-F_{1}$. Let $w$ be a neighbor of $u_{1}$ on $C_{1}$,
then $M_{0}\cup\{uw_{0}\}$ is a linear forest of at most $2(n-1)-3$
edges. By Lemma \ref{forest}, there exists a Hamiltonian cycle
$C_{0}$ containing $M_{0}\cup\{uw_{0}\}$ in $Q_{n-1}^{0}$. Then the
desired Hamiltonian cycle in $Q_{n}-F$ is induced by edges of
$(E(C_{0})\cup
E(C_{1})\cup\{uu_{1},w_{0}w\})\setminus\{uw_{0},u_{1}w\}$.
\end{proof}

%%%%%%%%%%%%%%%%%%%%%%%%%%%%%%%%%%%%%%%%%%%%%%%%%%%%%%%%%%%%%%%%
\vskip 0.2cm

\noindent{\large\bf Acknowledgements}

\vskip 0.2cm

\noindent The authors would like to express their gratitude to the
anonymous referees for their kind suggestions and careful
corrections on the original manuscript.

%%References----------------------------------------------------

\end{document}